\documentclass[reqno]{amsart} 

\usepackage{graphicx, amsmath, amssymb, amsfonts, amsthm, stmaryrd, amscd}
\usepackage[usenames, dvipsnames]{xcolor}

\usepackage{enumerate}
\usepackage[normalem]{ulem}
\usepackage{ifpdf}
\ifpdf \RequirePackage[pdftex]{hyperref} \fi
\usepackage{xparse}

\usepackage[all]{xy}
\usepackage{enumerate}

\makeatletter
\newcommand*{\transpose}{%
  {\mathpalette\@transpose{}}%
}
\newcommand*{\@transpose}[2]{%
  \raisebox{\depth}{$\m@th#1\intercal$}%
}
\makeatother

\makeatletter
\newcommand*{\da@rightarrow}{\mathchar"0\hexnumber@\symAMSa 4B }
\newcommand*{\da@leftarrow}{\mathchar"0\hexnumber@\symAMSa 4C }
\newcommand*{\xdashrightarrow}[2][]{%
  \mathrel{%
    \mathpalette{\da@xarrow{#1}{#2}{}\da@rightarrow{\,}{}}{}%
  }%
}
\newcommand{\xdashleftarrow}[2][]{%
  \mathrel{%
    \mathpalette{\da@xarrow{#1}{#2}\da@leftarrow{}{}{\,}}{}%
  }%
}
\newcommand*{\da@xarrow}[7]{%
  \sbox0{$\ifx#7\scriptstyle\scriptscriptstyle\else\scriptstyle\fi#5#1#6\m@th$}%
  \sbox2{$\ifx#7\scriptstyle\scriptscriptstyle\else\scriptstyle\fi#5#2#6\m@th$}%
  \sbox4{$#7\dabar@\m@th$}%
  \dimen@=\wd0 %
  \ifdim\wd2 >\dimen@
    \dimen@=\wd2 %   
  \fi
  \count@=2 %
  \def\da@bars{\dabar@\dabar@}%
  \@whiledim\count@\wd4<\dimen@\do{%
    \advance\count@\@ne
    \expandafter\def\expandafter\da@bars\expandafter{%
      \da@bars
      \dabar@ 
    }%
  }%  
  \mathrel{#3}%
  \mathrel{%   
    \mathop{\da@bars}\limits
    \ifx\\#1\\%
    \else
      _{\copy0}%
    \fi
    \ifx\\#2\\%
    \else
      ^{\copy2}%
    \fi
  }%   
  \mathrel{#4}%
}
\makeatother

\usepackage{mathtools}

\DeclareMathOperator{\SL}{SL}

\DeclareMathOperator{\GL}{GL}

\DeclareMathOperator{\SO}{SO}

\DeclareMathOperator{\ad}{ad}

\def\eps{\varepsilon}

\DeclareMathOperator{\PGL}{PGL}

\DeclareMathOperator{\trace}{trace}

\def\O{\operatorname{O}}

\theoremstyle{plain} \newtheorem{theorem} {Theorem}   \newtheorem{proposition} [theorem] {Proposition} 

\theoremstyle{definition}   
\newtheorem{example} [theorem] {Example}

            \newtheorem{remark} [theorem] {Remark}

\newtheoremstyle{itplain} % name
{6pt}                    % Space above
{5pt\topsep}                    % Space below
{\itshape}                   % Body font
{}                           % Indent amount
{\itshape}                   % Theorem head font
{.}                          % Punctuation after theorem head
{5pt plus 1pt minus 1pt}                       % Space after theorem head
{}  % Theorem head spec (can be left empty, meaning ‘normal’)

\theoremstyle{itplain} %--default
\newtheorem{lemma}[theorem]{Lemma}

\newtheorem*{lemma*}{Lemma}
\newtheorem*{proposition*}{Proposition}
\newtheorem*{definition*}{Definition}
\newtheorem*{example*}{Example}

\newtheorem*{results*}{Results}

\usepackage[displaymath,textmath,sections,graphics]{preview}
\PreviewEnvironment{align*}
\PreviewEnvironment{multline*}
\PreviewEnvironment{tabular}
\PreviewEnvironment{verbatim}
\PreviewEnvironment{lstlisting}
\PreviewEnvironment*{frame}
\PreviewEnvironment*{alert}
\PreviewEnvironment*{emph}
\PreviewEnvironment*{textbf}
\PreviewEnvironment*{pn}

\begin{document}

\author{Paul D. Nelson}

% \address{Aarhus University, Denmark}

\address{Aarhus University, Ny Munkegade 118,
  DK-8000 Aarhus, Denmark}
\email{paul.nelson@math.au.dk}


\title[The subconvexity-QUE implication for $\GL_2$]{Soft bounds for local triple products and the subconvexity-QUE implication for $\mathrm{GL}_2$}

\thanks{The author is supported by a research grant (VIL54509) from VILLUM FONDEN}

\begin{abstract}
  We give a soft proof of a uniform upper bound for the local factors in the triple product formula, sufficient for deducing effective and general forms of quantum unique ergodicity (QUE) from subconvexity.
\end{abstract}

\maketitle

\section{Introduction}\label{sec:cq72y61z2j}

\subsection{Overview}\label{sec:cq72zil9ow}

There are by now many results relating quantum unique ergodicity for $\GL_2$ automorphic forms to subconvex bounds for triple product $L$-functions.  The source of the relationship is the triple product formula, established in a general form by Kudla--Harris \cite{harris-kudla-1991} and Ichino \cite{MR2449948}.  A special case is that for $L^2$-normalized even Hecke--Maass cusp forms $\varphi_1, \varphi_2, \varphi_3$ on $\SL_2(\mathbb{Z})$, there is a \emph{local factor} $c_\infty \geq 0$, given by an archimedean integral, so that the following relation holds between triple product integrals and values of triple product $L$-functions (excluding archimedean factors):
\begin{equation}\label{eq:cq20knv1tn}
  \left\lvert   \int_{\SL_2(\mathbb{Z}) \backslash \mathbb{H}}
    \varphi_1 \varphi_2 \varphi_3  \, \frac{d x \, d y}{y ^2}\right\rvert^2
  = \frac{1}{8}
  c_\infty \frac{L(\tfrac{1}{2}, \varphi_1 \times \varphi_2 \times \varphi_3)}{\prod_j L(1, \ad \varphi_j)}.
\end{equation}
The general case applies to factorizable vectors $\varphi_1, \varphi_2, \varphi_3$ in automorphic representations $\pi_1, \pi_2, \pi_3$ on $\GL_2$ over a number field, and involves a product of local factors $c_{\mathfrak{p}}$ indexed by finitely many places.  Knowing the sizes of the local factors allows us to relate bounds for triple product integrals to bounds for triple product $L$-values.

Watson \cite[Theorem 3]{watson-2008} determined the local factor $c_\infty$ in \eqref{eq:cq20knv1tn} precisely, showing that it is equal to the corresponding ratio of archimedean $L$-factors.  The case relevant for quantum unique ergodicity is when $\varphi_3$ is fixed, $\varphi_2 = \overline{\varphi_1}$, and the eigenvalue $\lambda$ of $\varphi_1$ tends to infinity.  In that case, bounds for the integrals
\begin{equation}\label{eq:cq20065vcq}
  \int_{\SL_2(\mathbb{Z}) \backslash \mathbb{H}}
  \lvert \varphi_1 \rvert^2 \varphi_3 \, \frac{d x \, d y}{y^2},
\end{equation}
together with their counterparts when $\varphi_3$ is Eisenstein, capture the $L^2$-mass distribution of $\varphi_1$.  Watson's formula and Stirling's asymptotics yield $c_\infty \asymp \lambda^{- 1/2}$, while the convexity or ``trivial'' bound for the ratio of $L$-functions in \eqref{eq:cq20knv1tn} is $\O(\lambda^{1/2 + \eps})$, so any ``subconvex'' improvement $\O(\lambda^{1/2 - \delta})$ in the exponent translates to an effective rate of equidistribution for the $L^2$-mass of $\varphi_1$.  Such improvements remain unavailable outside special cases \cite{Sar01}, but equidistribution with an ineffective rate has meanwhile been proved by other means \cite{MR2195133, MR2680500}; see also \cite{MR2680499}.

Luo--Sarnak \cite{MR1361757} applied a similar relationship when $\varphi_1$ is an Eisenstein series.  Many authors have since established further relationships: for the cotangent bundle of a hyperbolic surface \cite{MR1313792, MR4879363}, the setting of Hecke eigenforms on the sphere \cite{MR2018269}, hyperbolic $3$-folds \cite{MR3219586}, and in level aspects \cite{PDN-HQUE-LEVEL, PDN-AP-AS-que, nelson-padic-que, MR3801500, HNSminimal}.  A key step in each of these works has been to determine (or estimate) the local factors, using various tools: formulas for Mellin transforms of products of Bessel functions, hypergeometric identities, properties of epsilon factors of supercuspidal representations, and so on.

The purpose of this note is to give a uniform upper bound (Theorem~\ref{theorem:cq2y88479a}) for the local factors, valid in all aspects (eigenvalue, weight, level, etc.).  Roughly speaking, in the above setting, we show that
\begin{equation*}
  \prod_{\mathfrak{p}} c_{\mathfrak{p}} \ll_{\varphi_3} C(\pi_1 \times \tilde{\pi}_1)^{- 1/4} C(\pi_2 \times \tilde{\pi}_2)^{- 1/4}
\end{equation*}
with $C(\dotsb)$ the analytic conductor (see \S\ref{sec:cq20musxtb}) and with explicit dependence upon the $K$-type of $\varphi_3$ (polynomial or better).  From this estimate, it is immediate that subconvexity in the $\pi_1$-aspect for $L(\tfrac{1}{2}, \pi_1 \times \tilde{\pi}_1 \times \pi_3)$ implies general effective forms of quantum unique ergodicity (see e.g.\ Theorem~\ref{theorem:cq2018exw2}).

The bounds obtained here are new in some classical cases, e.g., for the generalization of \eqref{eq:cq20knv1tn} to weight vectors with all three weights nonzero (compare with \cite{watson-2008, MR2982336, 2013arXiv1303.6972S, 2018arXiv180609767C, MR4193477, MR4879363}).  Another new feature is the uniformity of our estimates with respect to the $K$-type of the observable ($\varphi_3$ in the above discussion), avoiding the exponential loss present in previous works (see \cite[p.~969]{MR1465794} and \cite[Remarks 10.1 and A.8]{MR4879363}) and allowing applications to smooth observables in settings previously limited to $K$-finite ones (see Example~\ref{example:cq2010s6bn} and \S\ref{sec:cq72y5it3a}).

We do not evaluate the $c_{\mathfrak{p}}$ exactly.  This loss in precision is compensated by a gain in simplicity, as the proof uses only standard representation-theoretic tools: matrix coefficient bounds \cite{MR946351}, a positivity argument involving the complementary series, a linearization identity for matrix coefficient integrals \cite[Lemma 3.4.2]{michel-2009}, and the convexity principle for local Rankin--Selberg integrals (\cite[\S2.4]{PDN-AP-AS-que}, generalized from $\mathbb{Q}_p$ to all local fields).  While these ingredients are familiar, we hope it will be useful to record how they combine to give a soft deduction of quantum unique ergodicity from subconvexity.

We describe our main results and applications in more detail in the remainder of \S\ref{sec:cq72y61z2j}.  The rest of the paper is devoted to the proof of Theorem \ref{theorem:cq2y88479a}, following the outline given in \S\ref{sec:cq72y60xgx}.

\subsection{Notation and setup}\label{sec:cq72y62p45}

We begin with some notation.  The asymptotic notations $A = \O(B)$, $A \ll B$ and $B \gg A$ mean that $\lvert A \rvert \leq C \lvert B \rvert$ for some absolute constant $C$.  We use subscripted parameters, like in $A \ll_\eps B$, to indicate that $C$ may depend on those parameters.

For each local field $F$, we equip $\PGL_2(F)$ with a Haar measure in a standard way, recorded below in \S\ref{sec:cq2028ta9m}.  For instance, when $F$ is non-archimedean, this measure assigns volume one to each maximal compact subgroup.  We write $K$ for the standard maximal compact subgroup of $\GL_2(F)$ or $\PGL_2(F)$, according to context.  Given a vector $u$ in a representation of $\GL_2(F)$, we set
\begin{equation*}
  \dim K u := \text{dimension of the span of the $K$-orbit of $u$.}
\end{equation*}

  We now let $F$ be a number field, and denote by $\mathbb{A}$ its adele ring.  We equip $[\PGL_2] := \PGL_2(F) \backslash \PGL_2(\mathbb{A})$ with the probability Haar measure.  Let $\pi$ be a unitary cuspidal automorphic representation for $\GL_2$ over $F$ (see e.g.\ \cite{MR1431508}).  We equip it with the Petersson inner product, given by integration over $[\PGL_2]$.  It factors as a restricted tensor product, over the places $\mathfrak{p}$ of $F$, of generic irreducible unitary representations $\pi_{\mathfrak{p}}$ of $\GL_{2}(F_{\mathfrak{p}})$ -- the ``local components'' of $\pi$.  We may speak of factorizable vectors $\varphi = \otimes_{\mathfrak{p}} \varphi_{\mathfrak{p}}$, whose matrix coefficients likewise factor as $\langle g \varphi, \varphi \rangle = \prod_{\mathfrak{p}} \langle g_{\mathfrak{p}} \varphi_{\mathfrak{p}}, \varphi_{\mathfrak{p}} \rangle$.  The generalized Ramanujan conjecture predicts that each local component $\pi_\mathfrak{p}$ is tempered; towards that, it is known \cite{MR2811610} that they are $\vartheta$-tempered (see \S\ref{sec:cq2zo3v0dy} for definitions) with $\vartheta \leq 7/64$.

  We now recall the triple product formula in the form given by Ichino \cite{MR2449948}.  Let $\pi_1, \pi_2, \pi_3$ be unitary cuspidal automorphic representations for $\GL_2$ over $F$ whose product of central characters is trivial.  Let $\varphi_1, \varphi_2, \varphi_3$ be factorizable vectors, thus $\varphi_j = \otimes_{\mathfrak{p}} \varphi_{j, \mathfrak{p}}$ with $\varphi_{j, \mathfrak{p}} \in \pi_{j, \mathfrak{p}}$.  Let $S$ be a large enough finite set of places of $F$, containing all archimedean places as well as any places at which any of our vectors are ramified.  Ichino's formula asserts that
  \begin{equation}\label{eq:cq201cgwq7}
  \left\lvert   \int_{[\PGL_2]}
    \varphi_1 \varphi_2 \varphi_3 \right\rvert^2
  = c_{F}^{(S)}
  \left( \prod_{\mathfrak{p} \in S} c_\mathfrak{p} \right)
  \frac{L^{(S)}(\pi_1 \times \pi_2 \times \pi_3, \tfrac{1}{2})}{\prod_{j=1}^{3} L^{(S)}(1, \ad \pi_j)}.
\end{equation}
Here $c_{F}^{(S)} \asymp_F 1$, $L ^{(S)}$ denotes the contribution of Euler factors outside $S$, and $c_{\mathfrak{p}}$ is the local integral of matrix coefficients
\begin{equation}\label{eq:cq201b36xn}
  c_{\mathfrak{p}} = \int_{\PGL_2(F_\mathfrak{p})} \Bigl( \prod_{j = 1,2,3} \langle g \varphi_{j, \mathfrak{p}}, \varphi_{j, \mathfrak{p}} \rangle \Bigr) \, d g.
\end{equation}
When some or all of the forms are Eisenstein, Rankin--Selberg unfolding gives a similar identity (for the regularized integral, in the case of three Eisenstein series) after replacing adjoint $L$-values by their residues and modifying $c_{F,S}$ slightly (see \cite[\S4.4]{michel-2009}).  The integrals \eqref{eq:cq201b36xn} converge in view of the noted $7/64$-temperedness (indeed, any bound $< 1/6$ would suffice).  The case relevant for quantum unique ergodicity is when $\pi_2 = \tilde{\pi}_1$.

\subsection{Local triple product estimates}

Our main result is the following general upper bound for the local factors $c_{\mathfrak{p}}$.  The ``$F$'' below should be thought of as one of the completions ``$F_\mathfrak{p}$'' arising above, and one can simply take $\vartheta_j = 7/64$.
\begin{theorem}\label{theorem:cq2y88479a}
  Let $\vartheta_1, \vartheta_2, \vartheta_3 \in [0,1/2)$ with $2\max(\vartheta_1, \vartheta_2) + \vartheta_3 < 1/2$.  Set
  \begin{equation*}
    \Theta := \frac{\vartheta_1 + \vartheta_2}{1 - 2 \vartheta_3} < 1/2,
  \end{equation*}
  and let $\eps >0$.  Let $F$ be a local field.  Let $\pi_1, \pi_2, \pi_3$ generic irreducible unitary representations of $\GL_2(F)$.  Assume that $\pi_j$ is $\vartheta_j$-tempered (see \S\ref{sec:cq2zo3v0dy}) and that the product of central characters is trivial.  Let $v_j, v_j' \in \pi_j$ be unit vectors.  Then, with $C(\dotsb)$ the analytic conductor (see \S\ref{sec:cq20musxtb}), we have
  \begin{equation}\label{eq:cq2zpduv15}
    \int_{\PGL_2(F)}
    \Bigl( \prod_{j = 1,2,3} \langle g v_j, v_j' \rangle \Bigr)
    \, d g
    \ll_{\vartheta_1, \vartheta_2, \vartheta_3, \eps}
    \frac{\left( \dim K v_3 \right)^{\Theta + \eps} \left( \dim K v_3' \right)^{\Theta + \eps}}{C(\pi_1 \times \tilde{\pi}_1)^{1/4}
      C(\pi_2 \times \tilde{\pi}_2)^{1/4}}.
  \end{equation}
\end{theorem}
We emphasize that the implied constant in \eqref{eq:cq2zpduv15} satisfies only the indicated dependencies; in particular, it is uniform in $F$.
\begin{example}\label{example:cq2010s6bn}
  Suppose $F = \mathbb{R}$, $\pi_2 = \tilde{\pi}_1$, $\pi_1$ (resp.\ $\pi_3$) is a principal series representation attached to an even Maass form of eigenvalue $1/4 + r^2$ (resp.\ $1/4 + t^2$), and $v_1, v_2, v_3$ are vectors of weights $k, 0, -k$, respectively, for some $k \in \mathbb{Z}$.  Let $c$ denote the left hand side of \eqref{eq:cq2zpduv15}.  We have $C(\pi_1 \times \tilde{\pi}_1) \asymp 1 + r^2$ and $\dim K v_3 = 1$, so Theorem \ref{theorem:cq2y88479a} gives
  \begin{equation}\label{eq:cq201zufzj}
    c \ll \frac{1}{(1 + \lvert r \rvert)}.
  \end{equation}
  By specializing to $k = 0$, this recovers the ``equidistribution from subconvexity'' implication in Watson's thesis.  Returning to the case of general $k$, let us compare \eqref{eq:cq201zufzj} with \cite[Corollary A.4]{MR4879363}, which gives
  \begin{equation}\label{eq:cq72zrtyqy}
    c \ll_k \frac{1}{(1 + \lvert 2 r - t \rvert)^{1/2}(1 + \lvert 2 r + t \rvert)^{1/2}(1 + \lvert t \rvert)}.
  \end{equation}
  The main regime of interest is when $r \rightarrow \infty$ with $t$ essentially fixed (e.g., $\lvert t \rvert \leq \tfrac{1}{3} \lvert r \rvert$), where the above simplifies to
  \begin{equation}\label{eq:cq201zuspg}
    c \ll_k \frac{1}{(1 + \lvert r \rvert)(1 + \lvert t \rvert)}.
  \end{equation}
  Note that \eqref{eq:cq201zufzj} is weaker than \eqref{eq:cq201zuspg} in its dependence on $t$, but is uniform in $k$.  By contrast, as explained in \cite[Remarks 10.1 and A.8]{MR4879363} (see also \cite[p.~969]{MR1465794}), the proof of \eqref{eq:cq201zuspg} yields an exponential dependence on $k$.  For this reason, the main result of \cite{MR4879363} -- that subconvexity implies effective quantum unique ergodicity for microlocal lifts -- was established only for $K$-finite observables.  Using \eqref{eq:cq201zufzj} and its analogue for $\sigma$ in the discrete series, we may extend that result to observables that are merely smooth (see \S\ref{sec:cq72y5it3a}).  Similarly, using Theorem \ref{theorem:cq2y88479a} at a complex place, one may establish an effective version of \cite[Theorem 3]{MR3219586} involving smooth observables rather than just $K$-finite ones.
\end{example}
\begin{remark}\label{remark:cq21z2qqie}
  In typical applications of Theorem \ref{theorem:cq2y88479a}, the representations $\pi_1$ and $\pi_2$ are tempered, so one can take $\Theta = 0$.  Indeed, $\pi$ is tempered provided that $C(\pi \times \tilde{\pi})$ is large enough (e.g., $> 1$ when $F$ is non-archimedean).  When $\pi_1$ and $\pi_2$ are square-integrable, the proof simplifies and one can take $\Theta = \eps = 0$.  To see this, we use the polarization identity and the Cauchy--Schwarz inequality to reduce to bounding $\int_{\PGL_2(F)} \left\lvert \langle g v, v \rangle \right\rvert^2 \, d g$ for $v$ a unit vector in the square-integrable irreducible representation $\pi$.  This last integral is the reciprocal of the formal degree of $\pi$, which is known to be $\asymp C(\pi \times \tilde{\pi})^{1/2}$, completing the proof.  One may verify this last estimate using local Rankin--Selberg integrals \cite{MR729755}.  The proof of Theorem \ref{theorem:cq2y88479a} uses local Rankin--Selberg integrals in a similar way, and so may be understood informally as a deformation of the formal degree calculation.
\end{remark}
\begin{remark}\label{remark:cq2467ogbw}
  Suppose $\pi_3$ is fixed and $\pi_1 = \tilde{\pi_2}$ varies.  In many cases (see e.g.\ \cite{MR4879363}, \cite[Lemma 4.4]{PDN-HQUE-LEVEL}, \cite[Remark 50]{nelson-padic-que}, \cite[Theorem 1.2]{HNSminimal}), the upper bound \eqref{eq:cq2zpduv15} is realized by some vectors -- typically those chosen ``minimally'' in some sense.  It fails to be sharp for more general vectors, e.g., those varying inside a fixed representation \cite{MR1859598}, \cite[\S3.6--3.7]{michel-2009} or (unbalanced) newvectors \cite{PDN-AP-AS-que}.
\end{remark}
\begin{remark}
  We briefly compare with \cite[\S3.5]{michel-2009}, which gives a soft proof of an upper bound formulated roughly as follows: if $F$ is non-archimedean and $\pi_1, \pi_2, \pi_3$ are all tempered, then the left hand side of \eqref{eq:cq2zpduv15} is majorized ``up to epsilons'' by $\lVert v_1 \rVert_{\mathfrak{U}} \lVert v_1' \rVert_{\mathfrak{U}} \lVert v_2 \rVert_{\mathfrak{U}} \lVert v_2' \rVert_{\mathfrak{U}}$, where $\lVert v \rVert_{\mathfrak{U}}^4 := \int_{K} \left\lvert \langle k v, v \rangle \right\rvert^2 \, d k$.  To apply this, we must estimate the norms $\lVert v \rVert_{\mathfrak{U}}$.  The simplest way we know how to do so for general $v$ is essentially as in the proof of Theorem \ref{theorem:cq2y88479a}.
\end{remark}

\subsection{Global triple product estimates}

By applying Theorem \ref{theorem:cq2y88479a} to each local factor $c_{\mathfrak{p}}$ in \eqref{eq:cq201cgwq7} and its Eisenstein analogue, we obtain the expected global consequence:
\begin{theorem}\label{theorem:cq2y87nn4h}
  Let $F$ be a number field.  Let $\pi$ be a unitary cuspidal automorphic representation of $\GL_2(\mathbb{A})$, equipped with the Petersson inner product.  Let $\sigma$ be a generic automorphic representation of $\PGL_2(\mathbb{A})$ that contributes to the spectral decomposition of $L^2$, thus $\sigma$ is either
  \begin{itemize}
  \item unitary cuspidal, equipped with the Petersson inner product, or
  \item a space of Eisenstein series obtained via normalized induction from a unitary Hecke character, equipped with the inner product given by integration in the induced model.
  \end{itemize}
  Let $\varphi_1, \varphi_2 \in \pi$ and $\Psi \in \sigma$ be factorizable unit vectors.  Let $S$ denote the set of places that are either archimedean or at which any of these vectors is ramified.  Then, for some $c \geq 1$ depending at most on $F$,
  \begin{equation}\label{eq:cq2zo308i9}
    \left\lvert   \int_{[\PGL_2]}
      \varphi_1 \overline{\varphi_2} \Psi  \right\rvert^2
    \leq
    c^{\lvert S \rvert + 1}
    \frac{\dim K \Psi}{C(\pi \times \tilde{\pi})^{1/2}}
    \frac{L(\tfrac{1}{2}, \pi \times \tilde{\pi} \times \sigma)}{L(1, \ad \pi)^2 L^*(1, \ad \sigma)}.
  \end{equation}
  Here $K$ is the standard maximal compact subgroup of $\PGL_2(\mathbb{A})$, $L(\dotsb)$ denotes the finite part of the $L$-function, and $L^\ast(\dotsb)$ indicates that a residue is taken when $\sigma$ is Eisenstein.
\end{theorem}
\begin{proof}
  The bound follows from Ichino's formula \cite{MR2449948}, its analogue for Eisenstein series \cite[\S4.4]{michel-2009} and Theorem \ref{theorem:cq2y88479a}.  In verifying the hypotheses of the latter, we may use the $\vartheta \leq 7/64 < 1/6$ bound of \cite{MR2811610}.  For cosmetic purposes, we have replaced each partial $L$-function $L^{(S)}$ by simply $L$.  Doing so has no effect on our estimate: the temperedness bounds for $\pi$ and $\sigma$ imply that the local $L$-factors at finite places in $S$ are each $\asymp 1$, so we may absorb them into the quantity $c$.  Finally, we have used that $\Theta < 1/2$ to simplify the dependence on $\dim K \Psi$.
\end{proof}

\begin{remark}
  In practice, one can often sharpen \eqref{eq:cq2zo308i9} by replacing $\dim K \Psi$ with $(\dim K \Psi)^\eps$, for the same reasons as in Remark~\ref{remark:cq21z2qqie}.
\end{remark}

\subsection{Subconvexity implies quantum unique ergodicity}\label{sec:cq73e71skw}
We obtain the following ``all aspects'' generalization of the implication given in Watson's thesis.

\begin{theorem}\label{theorem:cq2018exw2}
  Let $\pi_j$ be a sequence of unitary cuspidal automorphic representations of $\GL_2(\mathbb{A})$.  Assume that the set $S_{\pi_j}$ of places at which $\pi_j$ ramifies has cardinality bounded independent of $j$.  Assume that for some $A \geq 0$ and all $\sigma$ as in Theorem \ref{theorem:cq2y87nn4h}, we have
  \begin{equation}\label{eq:cq20154jaz}
    \frac{ L(\tfrac{1}{2}, \pi_j \times \tilde{\pi}_j \times \sigma) }{C(\pi_j \times \tilde{\pi}_j)^{1/2} C(\sigma)^A L(1, \ad \pi_j)^2} \rightarrow 0.
  \end{equation}
  Let $\varphi_j \in \pi_j$ be a sequence of factorizable unit vectors, unramified outside $S_{\pi_j}$.  Assume that each $\pi_j$ is non-dihedral, i.e., not isomorphic to any of its Hecke character twists.  Then the sequence $\lvert \varphi_j \rvert^2$ equidistributes with respect to Haar measure on $[\PGL_2]$: for each bounded measurable function $\Psi : [\PGL_2] \rightarrow \mathbb{C}$, we have
  \begin{equation}\label{eq:cq20155ph0}
    \int_{[\PGL_2]} \lvert \varphi_j \rvert^2 \Psi \rightarrow \int_{[\PGL_2]} \Psi.
  \end{equation}
\end{theorem}
\begin{proof}
  This follows from Theorem \ref{theorem:cq2y87nn4h} by a standard approximation argument and spectral expansion, as in e.g.\ \cite[\S3.6]{PDN-AP-AS-que}.  Here the ``non-dihedral'' assumption tells us that $\lvert \varphi_j \rvert^2$ is orthogonal to the nontrivial one-dimensional representations $\mathbb{C} \chi(\det)$ attached to nontrivial quadratic Hecke characters $\chi$ (compare with \cite[\S3.3]{PDN-HMQUE}).
\end{proof}
\begin{remark}
  In interpreting \eqref{eq:cq20154jaz}, we recall the adjoint $L$-value lower bound (see \cite{MR1067982}, \cite{HL94}, \cite[\S 2.9]{2009arXiv0904.2429B})
  \begin{equation*}
    L(1, \ad \pi, 1) \gg_\eps C(\pi \times \tilde{\pi})^{-\eps},
  \end{equation*}
  which implies that \eqref{eq:cq20154jaz} follows from any estimate
  \begin{equation}\label{eq:cq21je2q0c}
    L(\tfrac{1}{2}, \pi \times \tilde{\pi} \times \sigma) \ll C(\pi \times \tilde{\pi})^{1/2 - \delta} C(\sigma)^{\O(1)}
  \end{equation}
  with $\delta > 0$.  Such estimates are called ``subconvex in $\pi$, with polynomial dependence on $\sigma$'' in view of the following consequence of the convexity bound and the inductive nature of local $\gamma$-factors (see \S\ref{sec:cq20musxtb} and \cite[\S3]{MR1990380}):
  \begin{equation*}
    L(\tfrac{1}{2}, \pi \times \tilde{\pi} \times \sigma) \ll_{\eps} C(\pi \times \tilde{\pi} \times \sigma)^{1/4 + \eps}
    \ll C(\pi \times \tilde{\pi})^{1/4 + 2 \eps} C(\sigma)^{\O(1)}.
  \end{equation*}

\end{remark}

\begin{remark}
  Many variations on Theorem \ref{theorem:cq2018exw2} may be similarly derived from Theorem \ref{theorem:cq2y88479a}.  For instance, one can
  \begin{itemize}
  \item assume stronger decay estimates in \eqref{eq:cq20154jaz}, such as those that would follow from \eqref{eq:cq21je2q0c}, which would yield corresponding strengthenings of \eqref{eq:cq20155ph0},
  \item let the cardinality of $S_{\pi_j}$ grow, by strengthening \eqref{eq:cq20154jaz} slightly, or
  \item replace $\lvert \varphi_j \rvert^2$ with more general products $\varphi_j ' \overline{\varphi_j}$, like in the ``Wigner distribution'' approach to defining microlocal lifts (see \S\ref{sec:cq72y5it3a}) or the study of ``dissipations of correlations'' \cite{2021arXiv2112.01427, MR4855321, 2024arXiv2405.05249}.
  \end{itemize}
\end{remark}

\subsection{Effective quantum unique ergodicity for Hecke--Maass forms, with smooth observables}\label{sec:cq72y5it3a}
We elaborate here on Example \ref{example:cq2010s6bn}.  Let $\varphi$ be a Hecke--Maass cusp form on $M := \mathrm{SL}_2(\mathbb{Z}) \backslash \mathbb{H}$, of eigenvalue $\lambda$.  We normalize $\varphi$ to have $L^2$-norm one, so that its squared magnitude defines a probability measure $\mu_\varphi$ on $M$, given for $\Psi \in C_c(M)$ by $\mu_\varphi(\Psi) := \langle \varphi \Psi, \varphi \rangle$ in $L^2(M)$.  This describes the distribution of $\varphi$ on the ``configuration space'' $M$.

The finer ``phase space'' portrait of $\varphi$ is controlled by the \emph{Wigner distributions} (or \emph{microlocal lifts}, …) $\omega_\varphi$ on $X := \mathrm{SL}_2(\mathbb{Z}) \backslash \mathrm{SL}_2(\mathbb{R})$, defined as follows (see e.g.\ \cite[\S6.5]{nelson-variance-3}, \cite[\S2--3]{MR1859345}, \cite[\S1D]{MR4879363}).  Let $K := \SO(2)$.  Write $\lambda = s(1 - s)$ with $s \in \mathbb{C}$.  For $k \in \mathbb{Z}$, define the Maass forms $\varphi_k$ of weight $k$ recursively by requiring that
\begin{equation*}
  \varphi_0 := \varphi,
  \quad
  i X \varphi_k =(s + k) \varphi_{k + 1},
  \quad
  i Y \varphi_k =(s - k) \varphi_{k - 1},
\end{equation*}
where $X := \tfrac{1}{2 i} \left(
  \begin{smallmatrix}
    1&i\\
    i&-1 \\
  \end{smallmatrix}
\right)$ and  $Y := \tfrac{1}{2 i} \left(
  \begin{smallmatrix}
    1&-i\\
    -i&-1 \\
  \end{smallmatrix}
\right)$ are raising and lowering operators.  Then $\varphi_k$ is an eigenvector for $K$; we denote by $\chi_k : K \rightarrow \mathbb{C}^\times$ its eigenvalue.  We recall that a function $\Psi : X \rightarrow \mathbb{C}$ is \emph{$K$-finite} if its right translates under $K$ have finite-dimensional span, or equivalently, if $\Psi$ is a finite linear combination of eigenvectors for $K$.  The functional
\begin{equation*}
  \omega_\varphi : \{\text{$K$-finite } \Psi \in C_c^\infty(X) \}
  \rightarrow \mathbb{C} 
\end{equation*}
\begin{equation*}
  \omega_\varphi(\Psi) := \sum_{k \in \mathbb{Z}} \langle \varphi_0 \Psi, \varphi_k \rangle
\end{equation*}
extends continuously to $C_c^\infty(X)$ and has the following properties:
\begin{itemize}
\item (Lifting) For each $\Psi \in C_c^\infty(M)$ (pulled back via $X \rightarrow M$, $g \mapsto g i$), we have $\omega_\varphi(\Psi) = \mu_\varphi(\Psi)$.
\item (Geodesic flow invariance) For each diagonal element $a \in \mathrm{SL}_2(\mathbb{R})$ and $\Psi \in C_c^\infty(X)$, we have $\omega_\varphi(\Psi(\bullet a)) - \omega_\varphi(\Psi) \rightarrow 0$ as $\lambda \rightarrow \infty$.
\end{itemize}

Let $\mu$ denote the invariant probability measure on $X$.  Theorem \ref{theorem:cq2y88479a} and a standard spectral argument (sketched below) yield the following:
\begin{theorem}[Subconvexity implies effective AQUE in phase space for smooth observables]\label{theorem:cq73e7xf3s}
  Assume subconvexity in the form \eqref{eq:cq21je2q0c}.  There exists $\delta > 0$ so that for each $\Psi \in C_c^\infty(X)$, there exists $B \geq 0$ so that for all $\varphi$ as above, we have
  \begin{equation}\label{eq:cq72zyf80f}
    \left\lvert \omega_\varphi(\Psi) - \mu(\Psi) \right\rvert \leq B \lambda^{- \delta}.
  \end{equation}
\end{theorem}
Theorem \ref{theorem:cq73e7xf3s} had previously been known in many cases.  The simplest case is when $\Psi$ is $K$-invariant, so that it factors through the ``configuration space'' $M$; then $\omega_\varphi(\Psi) = \mu_\varphi(\Psi)$, and Theorem \ref{theorem:cq73e7xf3s} follows from Watson's thesis.  If $\Psi$ is $K$-finite, then Theorem \ref{theorem:cq73e7xf3s} is given by \cite[Theorem 1.1]{MR4879363}.  The proof here mildly refines that in \emph{op.\ cit.}, so we will be brief.
\begin{proof}
  By Fourier analysis on $K$, we may decompose $\Psi$ as an infinite sum $\sum_{k \in \mathbb{Z}} \Psi_k$ of weight vectors, where each fixed Sobolev norm of $\Psi_k$ decays faster than any polynomial in $k$ (but no faster, in general).  Using the spectral decomposition for $L^2(X)$, we may further decompose each $\Psi_k$ as a sum/integral of Hecke--Maass forms of weight $k$, where the coefficients decay faster than any polynomial in $k$ and the Casimir eigenvalue, and the number (or volume) of terms up to a given eigenvalue is likewise polynomial.  Using this expansion, we reduce to establishing \eqref{eq:cq72zyf80f} when $\Psi$ is a non-constant spectrally normalized Hecke--Maass form (cuspidal or unitary Eisenstein) of weight $k$, provided that we can bound the quantity $B$ polynomially in the Casimir eigenvalue $\lambda_\Psi$ and weight $k$ of $\Psi$:
  \begin{equation}\label{eq:cq73fwimu8}
    B \ll(1 + \lvert \lambda_\Psi \rvert + \lvert k \rvert)^{\O(1)}.
  \end{equation}
  Since $\Psi$ has weight $k$, the Wigner distribution simplifies to $\omega_\varphi(\Psi) = \langle \varphi_0 \Psi, \varphi_k \rangle$.  By the triple product formula and our subconvexity hypothesis, we have
  \begin{equation*}
    \left\lvert \langle \varphi_0 \Psi, \varphi_k \rangle \right\rvert^2 \ll c \lambda^{1/2- \delta},
  \end{equation*}
  where $\delta > 0$ is fixed and $c$ is the local factor at $\infty$ in the triple product formula, given by an integral of matrix coefficients like on the left hand side of \eqref{eq:cq2zpduv15}.  Our task \eqref{eq:cq73fwimu8} reduces to showing that
  \begin{equation}\label{eq:cq72z04ml5}
    c \ll \lambda^{-1/2} (1 + \lvert \lambda_\Psi \rvert + \lvert k \rvert)^{\O(1)}.
  \end{equation}
  By construction, $\varphi_0$ and $\Psi$ are unit vectors, while a standard calculation with raising/lowering operators (e.g., \cite[\S7.2.1]{nelson-variance-3}) shows that $\lVert \varphi_q \rVert = 1$.  The required estimate follows from Theorem \ref{theorem:cq2y88479a}, which, as noted in Example \ref{example:cq2010s6bn}, gives
  \begin{equation*}
    c \ll \lambda^{-1/2}.
  \end{equation*}
\end{proof}
\begin{remark}
  For the proof given in \cite{MR4879363} of the $K$-finite case of Theorem \ref{theorem:cq73e7xf3s}, one may reduce to the case that $\Psi$ has some fixed weight $k$.  One must then show, in place of \eqref{eq:cq72z04ml5}, that
  \begin{equation}\label{eq:cq72zr0tte}
    c \ll_k \lambda^{-1/2} (1 + \lvert \lambda_\Psi \rvert)^{\O(1)}.
  \end{equation}
  This is obtained (with sharper dependence upon $\lambda_\Psi$) in \cite{MR4879363} (see, e.g., the estimates \eqref{eq:cq72zrtyqy} and \eqref{eq:cq201zuspg} quoted above).  The proof begins by evaluating $c$ in terms of generalized hypergeometric functions, such as ${}_4 F_3$, which may be expressed in terms of oscillating sums of ratios of $\Gamma$-functions (see \cite[Appendix A]{MR4879363}).  Bounding those sums using the triangle inequality and optimally estimating the magnitude of each term using Stirling's formula gives \eqref{eq:cq72zr0tte}.  The resulting bound depends exponentially upon $k$ \cite[Remarks 10.1 and A.8]{MR4879363}, and thus falls short of \eqref{eq:cq72z04ml5}.  (A similar exponential dependence appears in the earlier work of Jakobson \cite[p.~969]{MR1465794}.) To improve that dependence via the method of \emph{op.\ cit.}\ would require detecting cancellation in the sums defining the relevant hypergeometric functions.  Such cancellation is expected (indeed, it follows indirectly from Theorem \ref{theorem:cq2y88479a}), but seems difficult to detect directly while retaining uniformity in other parameters, such as $\lambda$.
\end{remark}

\subsection{Reduction of the proof}\label{sec:cq72y60xgx}

We turn to the proof of our main local result, Theorem \ref{theorem:cq2y88479a}.  By the Cauchy--Schwarz inequality and the polarization identity, it is enough to verify the following:
\begin{proposition}\label{proposition:cq21zxwnpk}
  Let $\vartheta_1, \vartheta_3 \in [0,1/2)$ with $\Theta := \frac{2 \vartheta_1}{1 - 2 \vartheta_3} < 1/2$, and let $\eps > 0$.  Let $F$ be a local field.  Let $\pi$ and $\sigma$ be generic irreducible unitary representations of $\GL_2(F)$.  Assume that $\pi$ is $\vartheta_1$-tempered and $\sigma$ is $\vartheta_3$-tempered.  Let $v \in \pi$ and $u_1, u_2 \in \sigma$ be unit vectors.  Set $d_j := \dim K u_j$.  Then
  \begin{equation*}
    \int_{\PGL_2(F)}
    \left\lvert \langle g v, v \rangle \right\rvert^2 \left\lvert \langle g u_1, u_2 \rangle \right\rvert \, d g
    \ll_{\Theta, \eps} \frac{(d_1 d_2)^{\Theta + \eps}}{C(\pi \times \tilde{\pi})^{1/2}}.
  \end{equation*}
\end{proposition}
The proof of Proposition \ref{proposition:cq21zxwnpk} is organized as follows.  In \S\ref{sec:cq20059ket}, we apply standard estimates for $\langle g u_1, u_2 \rangle$ to reduce to the case that $\sigma$ belongs to the unramified complementary series and $u_1 = u_2 =: u$ is $K$-invariant.  In that case, the matrix coefficient $\langle g u, u \rangle$ is nonnegative, so we may drop the absolute values.  In \S\ref{sec:cq2006b97i}, we apply an identity of Michel--Venkatesh to relate such matrix coefficient integrals to products of local Rankin--Selberg integrals.  In \S\ref{sec:cq2006fllf}, we estimate the latter integrals via the Phragmén--Lindelöf convexity principle, with the analytic conductor entering via the Stirling bound for the local $\gamma$-factor.  The proof is then complete.

\section{Preliminaries}

\subsection{Notation}\label{sec:cq72yg9xkn}
From now on, let $F$ be a local field.  We denote by $\lvert . \rvert : F \rightarrow \mathbb{R}_{\geq 0}$ the normalized absolute value.  When $F$ is non-archimedean, we denote by $\mathfrak{o}$ the ring of integers, $\mathfrak{p}$ its maximal ideal, and $q$ the size of $\mathfrak{o} / \mathfrak{p}$.  For $x \in F$ and $y \in F^\times$, we set
\begin{equation*}
  n(x) :=
  \begin{pmatrix}
    1    & x \\
    0 & 1 \\
  \end{pmatrix}, \quad
  a(y) :=
  \begin{pmatrix}
    y    & 0 \\
    0 & 1 \\
  \end{pmatrix},
  \quad 
  w :=
  \begin{pmatrix}
    0    & 1 \\
    1 & 0 \\
  \end{pmatrix}.
\end{equation*}
These will be understood as elements of $\GL_2(F)$ or $\PGL_2(F)$, according to context.  We let $K$ denote the standard maximal subgroup of $\GL_2(F)$ or $\PGL_2(F)$.  We write $\zeta_F(s)$ for the local zeta function, given by $\pi^{-s/2} \Gamma(s/2)$ or $2 (2 \pi)^{-s} \Gamma(s)$ or $(1 - q^{-s})^{-1}$ according as $F$ is real or complex or non-archimedean.  We retain the asymptotic notation recorded in \S\ref{sec:cq72y62p45}.

\subsection{Additive characters}\label{sec:cq203c622o}
Let $\psi$ be a nontrivial unitary character of $F$.  We assume that either $F$ is archimedean and $\psi(x) = e^{\pm 2 \pi i \trace_{F / \mathbb{R}}(x)}$, or $F$ is non-archimedean and $\psi$ is unramified, i.e., trivial on $\mathfrak{o}$ but not on $\mathfrak{p}^{-1}$.

\subsection{Measures}\label{sec:cq2028ta9m}
We equip $F$ with the $\psi$-self-dual Haar measure and $F^\times$ with the Haar measure $d^\times y = \zeta_F(1) \lvert y \rvert^{-1} \, d y$; in the non-archimedean case, each of these measures assigns volume one to maximal compact subgroups.  We transfer the measure on $F$ to
\begin{equation*}
  N := \left\{ n(x) : x \in F \right\},
\end{equation*}
which we understand as a subgroup of $\GL_2(F)$ or $\PGL_2(F)$ according to context.  We equip $K$ with the probability Haar measure $d k$ and $\PGL_2(F)$ with the Haar measure $d g = \lvert y \rvert^{-1} \, d x \, d^\times y \, d k$ given in terms of the Iwasawa decomposition $g = n(x) a(y) k$, $k \in K$ (i.e., $d g$ is the pushforward of the indicated measure under the map $(x, y, k) \mapsto n(x) a(y) k$ from $F \times F^\times \times K$ to $\PGL_2(F)$).  We equip $N \backslash \PGL_2(F)$ with the quotient measure, given explicitly by $d g = \lvert y \rvert^{-1} \,d^\times y \, d k$ for $g = a(y) k$.

\subsection{Spherical functions}
For each $s \in \mathbb{C}$, we let $\mathcal{I}(s)$ denote the corresponding induced representation of $\PGL_2(F)$, consisting of all smooth functions $f : \PGL_2(F) \rightarrow \mathbb{C}$ satisfying
\begin{equation*}
  f(n(x) a(y) g) = \lvert y \rvert^{s} f(g),
\end{equation*}
with $\PGL_2(F)$ acting via right translation: $g f(h) := f(h g)$.  We let
\begin{equation*}
  (,) : \mathcal{I}(s) \otimes \mathcal{I}(1-s) \rightarrow \mathbb{C}
\end{equation*}
denote the invariant pairing given by integration over $K$.  We let $f_s \in \mathcal{I}(s)$ denote the unique element satisfying
\begin{equation*}
  f_s \text{ is $K$-invariant,} \qquad f_s(1) = 1.
\end{equation*}
We let $\Xi_s : \PGL_2(F) \rightarrow \mathbb{C}$ denote the unique bi-$K$-invariant matrix coefficient of $\mathcal{I}(s)$ with $\Xi_s(1) = 1$, given explicitly by
\begin{equation*}
  \Xi_s(g) := (g f_s, f_{- s})
  = \int_{K} f_s(k g) f_{1-s}(k) \, d k = \int_{K} H(k g)^{s} \, d k,
\end{equation*}
where we write $H(g) := \lvert y \rvert$ using the Iwasawa decomposition $g \in n(x) a(y) K$ (in which $\lvert y \rvert$ is uniquely determined by $g$).  In particular, if $s \in \mathbb{R}$, then the integrand is positive, so $\Xi_s(g) > 0$ for all $g$.

\subsection{Temperedness}\label{sec:cq2zo3v0dy}
Let $\pi$ be a unitary representation of $\GL_2(F)$.  We recall that $\pi$ is \emph{tempered} if its matrix coefficients lie in $L^{2+\eps}$ modulo the center (see \cite{MR946351}).  For $\vartheta \in [0,1/2)$, we say that $\pi$ is \emph{$\vartheta$-tempered} if it does not weakly contain a complementary series representation of parameter strictly larger than $\vartheta$.  For example, when $\pi$ is irreducible, it is either
\begin{itemize}
\item square-integrable modulo the center, in which case it is tempered, or
\item an irreducible principal series representation, say $\sigma_1 \lvert . \rvert^{s_1} \boxplus \sigma_2 \lvert . \rvert^{s_2}$ with each $\sigma_i$ unitary, in which case it is $\vartheta$-tempered precisely when each $\lvert \Re(s_i) \rvert \leq \vartheta$.
\end{itemize}
In particular, ``$0$-tempered'' is equivalent to ``tempered''.  Throughout the paper, $\vartheta$ denotes an element of $[0, 1/2)$.

\subsection{Whittaker models}

Let $\pi$ be a generic irreducible unitary representation of $\GL_2(F)$.  Thus, $\pi$ admits a Whittaker model $\mathcal{W}(\pi, \psi)$, consisting of functions $W : \GL_2(F) \rightarrow \mathbb{C}$ satisfying $W(n(x) g) = \psi(x) W(g)$.  We also work with $\mathcal{W}(\pi, \psi^{-1})$, defined analogously.  We may normalize the inner product on $\pi$ to be given by integration over the diagonal, so that (see \cite[\S6.4]{MR748505} and \cite[\S10.2]{MR1999922})
\begin{equation*}
  \lVert W \rVert^2 = \int_{F^\times} \lvert W(a(y)) \rvert^2 \,d^\times y.
\end{equation*}
If $\pi$ is $\vartheta$-tempered, then for $y \in F^\times$ and $k \in K$, we have (see \cite[\S3.2.3]{michel-2009})
\begin{equation*}
  W(a(y) k) \ll_{W, \eps,M} \min \left( \lvert y \rvert^{1/2 - \vartheta - \eps}, \lvert y \rvert^{- M} \right)
\end{equation*}
for each $\eps > 0$ and $M < \infty$.

\section{Reduction to bounds for Rankin--Selberg integrals}

The purpose of this section is to reduce the proof of Theorem~\ref{theorem:cq2y88479a} to that of the bound \eqref{eq:cq2003eb3z} for local Rankin--Selberg integrals.

\subsection{Matrix coefficient bounds}\label{sec:cq20059ket}
\begin{lemma}\label{lemma:cq200vj1fw}
  Let $\vartheta < \alpha < 1/2$ and $\eps > 0$.  Let $\sigma$ be a $\vartheta$-tempered unitary representation of $\GL_2(F)$.  Let $u_1, u_2 \in \sigma$ be $K$-finite unit vectors.  Write $d_j = \dim K u_j$.  Then for each $g \in \GL_2(F)$,
  \begin{equation*}
    \left\langle g u_1, u_2 \right\rangle \ll_{\vartheta, \alpha, \eps}
    (d_1 d_2)^{\frac{1/2 - \alpha}{1 - 2 \vartheta} + \eps}
    \Xi_{1/2 + \alpha}(g).
  \end{equation*}
\end{lemma}
\begin{proof}
  By the Cartan decomposition, it suffices to consider the case $g = a(y)$, with $\lvert y \rvert \geq 1$.  In that case, \cite[Lemma 9.1]{venkatesh-2005} gives
  \begin{equation}\label{eq:cq21zymbdc}
    \left\langle a(y) u_1, u_2 \right\rangle \ll_\eps
    (d_1 d_2)^{1/2}
    \left( 1 + \lvert y \rvert \right)^{- 1/2 + \vartheta + \eps}.
  \end{equation}
  Interpolating between this and the trivial bound $\left\lvert \left\langle a(y) u_1, u_2 \right\rangle \right\rvert \leq 1$ yields, for $\eps$ small enough,
  \begin{align*}
    \left\langle a(y) u_1, u_2 \right\rangle
    &\ll_\eps
    \left( (d_1 d_2)^{1/2}
      \left( 1 + \lvert y \rvert \right)^{- 1/2 + \vartheta + \eps} \right)^{\frac{1/2 - \alpha + \eps}{1/2 - \vartheta - \eps}}
    \\
    &=
      (d_1 d_2)^{\frac{1/2 - \alpha}{1 - 2 \vartheta} + \eps '}(1 + \lvert y \rvert)^{- 1/2 + \alpha - \varepsilon},
  \end{align*}
  where $\eps ' \rightarrow 0$ as $\eps \rightarrow 0$.  As noted in \cite[(9.2)]{venkatesh-2005}, we have
  \begin{equation*}
    (1 + \lvert y \rvert)^{- 1/2 + \alpha - \eps} \ll_{\eps} 
    \Xi_{1/2 + \alpha}(a(y)).
  \end{equation*}
  Combining these last two estimates and taking $\eps$ small enough yields the required bound.
\end{proof}

For the proof of Proposition \ref{proposition:cq21zxwnpk}, we apply Lemma \ref{lemma:cq200vj1fw} with $\alpha := 1/2 - 2 \vartheta_1 - \eps$ and $\eps > 0$ sufficiently small.  Using the nonnegativity of $\Xi_s$ for $s$ real, we reduce to establishing the following:
\begin{proposition}\label{proposition:cq21zy73m0}
  Let $\pi$ be a $\vartheta$-tempered irreducible unitary representation of $\GL_2(F)$, let $v \in \pi$ be a unit vector, and let $s \in[2 \vartheta + \eps, 1 - 2 \vartheta - \eps]$.  Then
  \begin{equation}\label{eq:cq2002i2g9}
    \int_{\PGL_2(F)}
    \left\lvert \langle g v, v \rangle \right\rvert^2
    \Xi_{s}(g) \, d g
    \ll_{\eps} C(\pi \times \tilde{\pi})^{- 1/2}.
  \end{equation}
\end{proposition}

\subsection{Linearization of triple product periods}\label{sec:cq2006b97i}
Let $\pi$ be a $\vartheta$-tempered generic irreducible unitary representation of $\GL_2(F)$, and let $W \in \mathcal{W}(\pi, \psi)$.  The matrix coefficient integral
\begin{equation}\label{eq:cq2y8uea28}
  \int_{\PGL_2(F)}
  \left\lvert \langle g W, W \rangle \right\rvert^2
  \Xi_s(g)
  \, d g
\end{equation}
converges absolutely for $\Re(s) \in (2 \vartheta, 1 - 2 \vartheta)$, while the local Rankin--Selberg integral
\begin{equation}\label{eq:cq20mtt806}
  I(s,W) := \int_{N \backslash \PGL_2(F)} \lvert W(g) \rvert^2 f_s(g) \, d g
\end{equation}
converges absolutely for $\Re(s) > 2 \vartheta$.  Each integral defines a holomorphic function in the domain of definition.  The following result compares these functions:
\begin{lemma}\label{lemma:cq200vnl3f}
  For $\Re(s) \in (2 \vartheta, 1 - 2 \vartheta)$, we have
  \begin{equation*}
    \int_{\PGL_2(F)}
    \left\lvert \langle g W, W \rangle \right\rvert^2
    \Xi_s(g)
    \, d g
    =
    c_F I(s, W) I(1-s, W)
  \end{equation*}
  where $c_F \asymp 1$.
\end{lemma}
\begin{proof}
  Like in \cite[Prop 2.11]{PDN-AP-AS-que}, this follows from \cite[Lemma 3.4.2]{michel-2009}, first for $\Re(s) = 1/2$, then in general via analytic continuation.  (See also \cite[Lemma 6]{2021arXiv210112106B}, \cite[\S2.14]{nelson-subconvex-reduction-eisenstein}, \cite[Prop 5.1]{MR4234973}, \cite[Prop 4.5]{MR4193477}.)
\end{proof}

Lemma \ref{lemma:cq200vnl3f} reduces the proof of Proposition \ref{proposition:cq21zy73m0} to that of the estimate
\begin{equation*}
  I(s, W)
  I(1-s, W)
  \ll_{\eps} C(\pi \times \tilde{\pi})^{- 1/2}
\end{equation*}
whenever $\pi$ is $\vartheta$-tempered and $s \in[2 \vartheta + \eps, 1 - 2 \vartheta - \eps]$.  We may reduce further to showing that
\begin{equation}\label{eq:cq2003eb3z}
  I(s, W) \ll_\eps C(\pi \times \tilde{\pi})^{(s-1)/2}
\end{equation}
This estimate is given below in Proposition \ref{lemma:cq2003epv6}.

\section{Rankin--Selberg integrals}\label{sec:cq20mmeihl}
In this section, we record some standard background concerning local Rankin--Selberg integrals.  General references for this section are \cite[\S3]{MR1990380} and \cite{Ja72}.

\subsection{Zeta integrals}
We equip $\GL_2(F)$ with the Haar measure compatible with the chosen measures on $\PGL_2(F) = \GL_2(F) / F^\times$ and $F^\times$.

Let $\pi_1$ and $\pi_2$ be generic irreducible unitary representations of $\GL_2(F)$, with Whittaker models $\mathcal{W}(\pi_1, \psi)$ and $\mathcal{W}(\pi_2, \psi^{-1})$.  Assume that $\pi_i$ is $\vartheta_i$-tempered.  Let $\mathcal{S}(F^2)$ denote the space of Schwartz--Bruhat functions $\Phi : F^2 \rightarrow \mathbb{C}$.  Given $W_1 \in \mathcal{W}(\pi_1, \psi)$, $W_2 \in \mathcal{W}(\pi_2, \psi^{-1})$, $\Phi \in \mathcal{S}(F^2)$ and $s \in \mathbb{C}$, the local zeta integral
\begin{equation*}
  \Psi(s, W_1, W_2, \Phi) := \int_{N \backslash \GL_2(F)}
  W_1(g) W_2(g)
  \Phi(e_2 g) \lvert \det g \rvert^s \, d g,
  \quad
  e_2 := (0,1)
\end{equation*}
converges absolutely for $\Re(s) > \vartheta_1 + \vartheta_2$.  It extends to a meromorphic function on the complex plane, bounded at infinity in vertical strips.

\subsection{$L$-factors}\label{sec:cq72ygpj5r}

The local $L$-factor is of the form $L(s, \pi_1 \times \pi_2) = \prod_j \zeta_F(s - \mu_j)$, where the number of $\mu_j$ is at most $4$ and each satisfies
\begin{equation}\label{eq:cq203df2y5}
  \Re(\mu_j) \leq \vartheta_1 + \vartheta_2.
\end{equation}
The ratio
\begin{equation*}
  \frac{\Psi(s, W_1, W_2, \Phi)}{L(s, \pi_1 \times \pi_2)}
\end{equation*}
is entire.

\subsection{$\gamma$-factors and functional equation}

Set $\tilde{W}(g) := W(w g^{- \transpose})$ and $\hat{\Phi}(x) := \int_{y \in F^2} \Phi(y) \psi(y^{\transpose} x) \, d y$.  The local $\gamma$-factor is the meromorphic function for which the following local functional equation holds:
\begin{equation*}
  \Psi(1 - s, \tilde{W}_1, \tilde{W}_2, \hat{\Phi})
  = \omega_{\pi_2}(- 1) \gamma(s, \pi_1 \times \pi_2, \psi)
  \Psi(s, W_1, W_2, \Phi),
\end{equation*}
where $\omega_{\pi_i}$ denotes the central character of $\pi_2$.  We have
\begin{equation*}
  \gamma(s, \pi_1 \times \pi_2, \psi)
  =
  \eps(s, \pi_1 \times \pi_2, \psi)
  \frac{L(1-s, \tilde{\pi}_1 \times \tilde{\pi}_2)}{L(s, \pi_1 \times \pi_2)},
\end{equation*}
where the local $\eps$-factor is an entire function whose reciprocal is also entire.

\subsection{Conductors and Stirling asymptotics}\label{sec:cq20musxtb}
When $F$ is non-archimedean, we have
\begin{equation}\label{eq:cq20mtnvj7}
  \eps(s, \pi_1 \times \pi_2, \psi)
  =
  C(\pi_1 \times \pi_2)^{1/2 - s}
  \eps(\tfrac{1}{2}, \pi_1 \times \pi_2, \psi),
  \quad
  \lvert \eps(\tfrac{1}{2}, \pi_1 \times \pi_2, \psi) \rvert = 1,
\end{equation}
where $C(\pi_1 \times \pi_2)$, which we refer to as the \emph{conductor} of $\pi_1 \times \pi_2$, has the form $C(\pi_1 \times \pi_2) = q^{c(\pi_1 \times \pi_2)}$ for some $c(\pi_1 \times \pi_2) \in \mathbb{Z}_{\geq 0}$, the \emph{conductor exponent}.

When $F$ is archimedean, we define the \emph{analytic conductor} $C(\pi_1 \times \pi_2)$ as in e.g.\ \cite[(2.14)]{2021arXiv210112106B}.  Using the Legendre duplication formula $\zeta_{\mathbb{C}}(s) = \zeta_{\mathbb{R}}(s) \zeta_{\mathbb{R}}(s + 1)$, we may write $L(s, \pi_1 \times \pi_2) = \prod_{\nu} \zeta_{\mathbb{R}}(s - \nu)$ for some multiset $\{\nu\}$; then $C(\pi_1 \times \pi_2) \asymp \prod_{\nu}(1 + \lvert \nu \rvert)$.

\begin{lemma}
  Let $s \in \mathbb{C}$, $\eps > 0$ and $M \geq 0$ such that
  \begin{itemize}
  \item $\lvert s \rvert \leq M$, and
  \item $s$ is at least $\eps$ away from any pole of $L(s, \pi_1 \times \pi_2)$.
  \end{itemize}
  Then
  \begin{equation}\label{eq:cq20mtobvl}
    \frac{1}{\gamma(s, \pi_1 \times \pi_2, \psi)}
    \ll_{\eps, M} C(\pi_1 \times \pi_2)^{s - 1/2}.
  \end{equation}
\end{lemma}
\begin{proof}
  In the non-archimedean case, this follows from \eqref{eq:cq20mtnvj7}.  In the archimedean case, it is a consequence of Stirling's formula (see e.g.\ \cite[\S2.2]{2021arXiv210112106B}).
\end{proof}

\subsection{Specialization}
For a character $\chi : F^\times \rightarrow \mathbb{C}^\times$, define the local Tate integral
\begin{equation*}
  z(\chi, \Phi) := \int_{z \in F^\times} \chi(z) \Phi(e_2 z) \,d^\times z.
\end{equation*}
Set $g \Phi := \Phi(\bullet g)$, $\omega := \omega_{\pi_1} \omega_{\pi_2}$ the product of central characters, and $\alpha := \lvert . \rvert$.  As in \cite{Ja72}, we may write the quantities $\Psi(s, W_1, W_2, \Phi)$ and $\Psi(1 - s, \tilde{W}_1, \tilde{W}_2, \hat{\Phi})$, respectively, as
\begin{equation*}
  \int_{N \backslash \PGL_2(F)}
  W_1(g) W_2(g) z(\alpha^{2 s} \omega, g \Phi)
  \lvert \det g \rvert^{s} \, d g,
\end{equation*}
and
\begin{equation*}
  \int_{N \backslash \PGL_2(F)}
  W_1(g) W_2(g)
  z(\alpha^{2 s} \omega^{-1} , g \hat{\Phi}) \lvert \det g \rvert^s \omega^{-1}(\det g) \, d g.
\end{equation*}
Assume now that $\omega = 1$, and recall that $\psi$ is chosen as in \S\ref{sec:cq203c622o}.  There is a unique $K$-invariant $\Phi \in \mathcal{S}(F^2)$ that is also invariant under dilation by $\{z \in F^\times : \lvert z \rvert = 1\}$ and satisfies $z(\alpha^s, \Phi) = \zeta_F(s)$, given explicitly as follows
(see \cite[pp.~249--255]{MR1680912}):
\begin{equation*}
  \Phi(x, y) =
  \begin{cases}
    1_{\mathfrak{o}}(x) 1_{\mathfrak{o}}(y)    
    & \text{ if $F$ is non-archimedean}, \\
    e^{- \pi(x^2 + y^2)}
    & \text{ if } F = \mathbb{R}, \\
    (2 \pi)^{-1} e^{- 2 \pi (x \bar{x} + y \bar{y})}
                                                                                          & \text{ if } F = \mathbb{C}.
  \end{cases}
\end{equation*}
The same $\Phi$ satisfies $\hat{\Phi} = \Phi$ and
\begin{equation*}
  z(\alpha^{2 s}, g \Phi) = \zeta_F(2 s) f_s(g).
\end{equation*}
The condition $\omega = 1$ is satisfied when $\pi_2 = \tilde{\pi}_1$, in which case conjugation defines an anti-linear isomorphism $\mathcal{W}(\pi_1, \psi) \rightarrow \mathcal{W}(\pi_2, \psi^{-1})$.  We obtain the following:
\begin{lemma}\label{lemma:cq200uejip}
  Let $\pi$ be a $\vartheta$-tempered generic irreducible unitary representation of $\GL_2(F)$.  Let $W \in \mathcal{W}(\pi, \psi)$.  The normalized variant
  \begin{equation*}
    I^\ast(s, W) := \zeta_F(2 s) I(s,W)
  \end{equation*}
  of the integral defined in \eqref{eq:cq20mtt806} is bounded at infinity in vertical strips, the ratio
  \begin{equation*}
    \frac{I^\ast(s, W)}{L(s, \pi \times \tilde{\pi})}
  \end{equation*}
  is entire, and we have the functional equation
  \begin{equation}\label{eq:cq2027ojx6}
    I^\ast(1 - s, W) = \omega_{\pi}(-1) \gamma(s, \pi \times \tilde{\pi}, \psi) I^\ast(s, W).
  \end{equation}
\end{lemma}
% \begin{remark}
%   The unnormalized integral $I(s, W)$ is likewise bounded at infinity in vertical strips (see \cite[Lemma 7.3]{Nelson-EisCubic}).
% \end{remark}

\section{Bounds for Rankin--Selberg integrals}\label{sec:cq2006fllf}

Here we complete the proof of Theorem \ref{theorem:cq2y88479a} by establishing the remaining estimate \eqref{eq:cq2003eb3z}.  Throughout this section, let $\pi$ be a $\vartheta$-tempered generic irreducible unitary representation of $\GL_2(F)$, and let $W \in \mathcal{W}(\pi, \psi)$ be a unit vector.

\subsection{Using polynomials to cancel poles}
We define the multiset
\begin{equation*}
  \mathcal{Z} := \{\text{poles of $L(s, \pi \times \tilde{\pi})$ with $\Re(s) \geq -1$}\},
\end{equation*}
where, in the non-archimedean case, we work in the natural domain $\mathbb{C} / \frac{2 \pi i}{ \log q} \mathbb{Z}$.

\begin{lemma}
  The multiset $\mathcal{Z}$ is finite, with cardinality at most (say) $8$.
\end{lemma}
\begin{proof}
  As noted in \S\ref{sec:cq72ygpj5r}, the $L$-factor $L(s, \pi \times \tilde{\pi})$ is a product of at most four shifted zeta factors $\zeta_F(s - \mu)$, where $\Re(\mu) \leq 2 \vartheta < 1$.  A pole of this shifted factor with $\Re(s) \geq -1$ gives a pole of $\zeta_F$ with real part $> -2$.  By inspecting the definition of $\zeta_F$ (\S\ref{sec:cq72yg9xkn}), we see that the number of such poles is $1, 1$, or $2$ according as $F$ is non-archimedean, real, or complex.
\end{proof}

We define the function
\begin{equation*}
  P(s) := \prod_{\mu \in \mathcal{Z}}(1 - q^{-(s - \mu)}),
\end{equation*}
where (say) $q:= e$ for archimedean $F$.

\begin{lemma}\label{lemma:cq203dik2z}
  With notation as above:
  \begin{enumerate}[(i)]
  \item\label{enumerate:cq203dgehl} $P(s) \ll 1$ for all $s \in \mathbb{C}$.
  \item\label{enumerate:cq203dgiqv} The product $P(s) L(s, \pi_1 \times \pi_2)$ and ratio $P(s) / \gamma(s, \pi_1 \times \pi_2, \psi)$ are holomorphic for $\Re(s) \geq -1$.
  \item\label{enumerate:cq203dgnvc} $P(s) \gg_\eps 1$ if $s \in [2 \vartheta + \eps,1]$.
  \item\label{enumerate:cq203dguiz} For $\Re(s) \geq 0$, we have
    \begin{equation}\label{eq:cq200upyvy}
      \frac{P(s) }{ \gamma(s, \pi \times \tilde{\pi}, \psi)} \ll C(\pi \times \tilde{\pi})^{- 1/2},
    \end{equation}
  \end{enumerate}
\end{lemma}
\begin{proof}
  \eqref{enumerate:cq203dgehl} and \eqref{enumerate:cq203dgiqv} are clear by construction.  \eqref{enumerate:cq203dgnvc} follows from \eqref{eq:cq203df2y5}.  The Stirling bound \eqref{eq:cq20mtobvl} gives \eqref{enumerate:cq203dguiz}, first at least $\eps$-away from each pole of $L(s, \pi \times \tilde{\pi})$, then in general via Cauchy's formula.
\end{proof}

\subsection{Interpolation}
\begin{lemma}\label{lemma:cq2znfofde}
  For $\Re(s) \in [0,1]$, we have
  \begin{equation*}
    P(s) I^{\ast}(s, W) \ll C(\pi \times \tilde{\pi})^{(s-1)/2}.
  \end{equation*}
\end{lemma}
\begin{proof}
  By the construction of $P(s)$ and Lemma \ref{lemma:cq200uejip}, we know that $P(s) I^*(s,W)$ is holomorphic and bounded in the strip $\Re(s) \in[0, 1]$.  By the Phragmén--Lindelöf convexity principle (see \cite[Theorem 5.3]{MR2061214}), it suffices to establish the claimed bound when $\Re(s) \in \{0, 1\}$.  When $\Re(s) = 1$, the claim follows from the estimates $P(s) \asymp 1$, $\zeta_F(2 s) \asymp 1$ and
  \begin{align*}
    \lvert I(s,W) \rvert
    &\leq I(1,W)
      =
      \int_{F^\times}
      \int_{K}
      \lvert W(a(y) k) \rvert^2
      \, d k
      \,d^\times y \\
    &=
      \int_{k \in K}
      \lVert k W \rVert^2 \, d k
      =
      \lVert W \rVert^2 = 1.
  \end{align*}
  When $\Re(s) = 0$, the claim follows from the functional equation \eqref{eq:cq2027ojx6}, the bound $I^\ast(1 - s, W) \ll 1$ just established and the Stirling bound \eqref{eq:cq200upyvy}:
  \begin{equation*}
    P(s) I^\ast(s, W) = \frac{P(s)  I^\ast(1 - s, W)}{ \omega_{\pi}(-1)
      \gamma(s, \pi \times \tilde{\pi}, \psi)} \ll C(\pi \times \tilde{\pi})^{- 1/2}.
  \end{equation*}
\end{proof}

\subsection{Conclusion}\label{sec:cq21jjbnn2}
We now establish \eqref{eq:cq2003eb3z}, completing the proof of Theorem \ref{theorem:cq2y88479a}.

\begin{proposition}\label{lemma:cq2003epv6}
  For $\eps > 0$ and $s \in[2 \vartheta + \eps, 1]$, we have
  \begin{equation*}
    I(s, W) \ll_\eps C(\pi \times \tilde{\pi})^{(s-1)/2}.
  \end{equation*}
\end{proposition}
\begin{proof}
  This follows from Lemma \ref{lemma:cq2znfofde} upon noting that $P(s) \gg_\eps 1$ (Lemma \ref{lemma:cq203dik2z}, part \eqref{enumerate:cq203dgnvc}) and $\zeta_F(2 s) \gg_{\eps} 1$ (since $s \geq \eps$).
\end{proof}

\textbf{Acknowledgement.} We would like to thank Yueke Hu, Peter Humphries, Subhajit Jana, Andre Reznikov, Abhishek Saha, and the anonymous referee for helpful feedback on an earlier draft.

% \bibliography{refs}{} \bibliographystyle{alpha}

\newcommand{\etalchar}[1]{$^{#1}$}
\def\cprime{$'$} \def\cprime{$'$} \def\cprime{$'$} \def\cprime{$'$}

\end{document}